\documentclass[reqno]{amsart}

\usepackage{fullpage,amssymb,dsfont}

\usepackage[all]{xy}

\usepackage{hyperref}

\numberwithin{equation}{section}

\newtheorem{theorem}{Theorem}[section]
\newtheorem{corollary}[theorem]{Corollary}
\newtheorem{lemma}[theorem]{Lemma}
\newtheorem{proposition}[theorem]{Proposition}
\theoremstyle{remark}

\newtheorem{remark}[theorem]{Remark}

\theoremstyle{definition}

\newcommand{\bp}{\begin{proof}}
\newcommand{\ep}{\end{proof}}

\mathchardef\mhyph="2D

\newcommand{\C}{\mathbb{C}}

\newcommand\Q{\mathbb Q}
\newcommand{\Z}{\mathbb{Z}}

\newcommand\A{\mathcal{A}}
\newcommand{\CC}{\mathcal{C}}

\newcommand{\RR}{\mathcal{R}}

\newcommand{\OO}{\mathcal{O}}
\newcommand{\U}{\mathcal{U}}

\newcommand\g{\mathfrak{g}}
\newcommand\h{\mathfrak{h}}
\newcommand\kk{\mathfrak{k}}
\newcommand\sln{\mathfrak{sl}_n}

\newcommand\Qq{{\Q[q,q^{-1}]}}

\newcommand{\Pol}{\mathcal{O}}
\newcommand\Uq{{U_q{\mathfrak g}}}
\newcommand\Uqres{{U_q^{\mathrm{res}}{\mathfrak g}}}

\newcommand\ad{\operatorname{ad}}

\newcommand\Aut{\operatorname{Aut}}
\newcommand{\End}{\operatorname{End}}
\newcommand\Hilb{\operatorname{Hilb}}
\newcommand\Hom{\operatorname{Hom}}

\newcommand\Rep{\operatorname{Rep}}
\newcommand\Out{\operatorname{Out}}
\newcommand\Tr{\operatorname{Tr}}

\begin{document}

\title{Relations in quantized function algebras}

\date{September 4, 2017}

\author{Pavel Etingof}
\address{P.E.: Department of Mathematics, Massachusetts Institute of Technology,
Cambridge, MA 02139, USA}
\email{etingof@math.mit.edu}

\author{Sergey Neshveyev}
\address{S.N.: Mathematics institute,
University of Oslo,
Moltke Moes vei 35,
0851 Oslo, Norway}
\email{sergeyn@math.uio.no}

\begin{abstract} We develop a method to give presentations of quantized function algebras of complex reductive groups.
In particular, we give presentations of quantized function algebras of automorphism groups of finite dimensional simple complex Lie algebras.
\end{abstract}

\maketitle

\section{Introduction}

A reductive group can often be conveniently defined as the group of symmetries of a given collection of tensors on a finite dimensional vector space $V$. For example, the orthogonal, respectively symplectic, group is the group of symmetries of a symmetric, respectively skew-symmetric, nondegenerate inner product, and the special orthogonal group is the joint symmetry group of the inner product and the volume form. More interesting examples are the group $G_2$, which may be defined as the group of symmetries of the inclusion $V\to V\otimes V$, where $V$ is the $7$-dimensional irreducible representation, and $E_6$, which is the group of symmetries of the invariant cubic form on the $27$-dimensional irreducible representation $V$. Finally, if $\g$ is a simple Lie algebra, then
$\Aut(\g)$ is the group of symmetries of the bracket $\g\otimes \g\to \g$. This gives rise to presentations of the function algebras $\Pol(G)$ of such groups $G$ by generators and relations, where the generators are the matrix coefficients of $V$ or of $V$ and $V^*$, and the relations come from the condition that the given tensors are preserved.

The goal of this paper is to extend such presentations to the $q$-deformed setting. If $q=e^\hbar$, where $\hbar$
is a formal parameter, and we consider formal ($\hbar$-adically complete) quantum deformations $\Pol_q(G)$ of $\Pol(G)$, then such an extension is rather straightforward: we just need to deform the relations so that they are satisfied in the $q$-deformed case, and this will give a presentation of $\Pol_q(G)$ by a standard formal deformation argument. However, if we work over $\C(q)$ or with numerical $q$, this argument does not work and it is not clear why the deformed relations are sufficient to define $\Pol_q(G)$. Showing the sufficiency of the relations, or, in other words, that the algebra they define has at most the expected size, is normally called the ``easy part of the PBW theorem'', but in absence of suitable filtrations, as in our case, it is actually not so easy.

We propose a general method of proving the sufficiency of the deformed relations for generic $q$, i.e., outside of countably many complex values of $q$ or over $\C(q)$. This method is based on an idea from the theory of C$^*$-tensor categories. Namely, we replace the reductive group $G$ with the corresponding compact real form~$K$. Then we show that if $K$ is defined as the group of unitary transformations of $V$ preserving a collection of tensors~$T_i$, then the category $\Rep G=\Rep K$ is generated by $T_i$ and their (Hermitian) adjoints~$T_i^*$, the duality morphisms and their adjoints, and the braiding, i.e., any morphism in this category can be expressed through them. The main idea of proving this is that the Karoubian category generated by all these morphisms is in fact semisimple abelian because of unitarity, and hence by Tannakian formalism must be the category of representations of some closed subgroup $K'\subset U(V)$ containing $K$, and then it is not hard to show that in fact $K'=K$.

This allows us to argue that for generic $q$ the representation category
$\Rep G_q$ of the corresponding quantum group $G_q$ is still generated by appropriate $q$-deformations of the classical morphisms. We then get a presentation of $\Pol_q(G)$ as the Hopf algebra of quantum symmetries of the $q$-deformations of $T_i$ and $T_i^*$ for generic $q$ (and even over $\Q(q)$, if $T_i$ are defined over $\Q$), which is a generalization of the description of quantum groups of classical series due to Faddeev, Reshetikhin and Takhtadzhyan~\cite{FRT}. This presentation gives a somewhat larger set of relations than one wanted, but one can hope that with additional work some of these relations turn out to be redundant.

We implement this strategy in full detail in the case of $G=\Aut(\g)$, where $\g$ is a simple complex Lie algebra. Moreover, we use a certain noncommutative version of Grothendieck's generic freeness theorem, due to Artin, Small and Zhang~\cite{ASZ}, to show that our presentations of quantized function algebras are valid outside of finitely many values of $q$. This, in particular, implies that they are valid for almost all roots of unity.

We expect that these results can be extended without significant changes to many other examples, but leave this outside the scope of the paper.

\smallskip

The organization of the paper is as follows. In Section~\ref{s2} we outline the strategy. In Section~\ref{s3} we prove the main theorem. In Section~\ref{s4} we strengthen it so that it applies to all values of $q$ except finitely many, and give an application to classifying tensor autoequivalences.

\smallskip

{\bf Acknowledgements.} The authors thank Ken Brown and Milen Yakimov for useful discussions. The work of P.E. was partially supported by the NSF grant DMS-1502244. The work of S.N. was partially supported by the ERC grant no.~307663.

\section{The strategy}\label{s2}

Our starting point is the following simple consequence of the
Tannaka--Krein duality.

\begin{theorem}\label{thm:TKu}
Let $V$ be a finite dimensional complex Hilbert space. Consider the standard
representation of the unitary group $U(V)$ on $V$ and the
contragredient representation on the dual Hilbert space $V^*$.
Assume we are given a collection of tensors $T_i\in (V^*)^{\otimes
m_i}\otimes V^{\otimes n_i}$, $i\in I$, and let $K\subset U(V)$ be
the closed subgroup stabilizing all these tensors. Then the monoidal
category $\Rep K$ of finite dimensional complex representations of
$K$ is generated by $V$ and~$V^*$ and the following morphisms:
\begin{itemize}
\item[(a)] the flip $\sigma\colon V\otimes V\to
V\otimes V$;
\item[(b)] the morphisms $r\colon \C\to V^*\otimes V$,
$r(1)=\sum_k e^k\otimes e_k$, and $\bar r\colon\C\to V\otimes V^*$,
$\bar r(1)=\sum_ke_k\otimes e^k$, and their adjoints, where
$\{e_k\}_k$ is a basis in $V$ and $\{e^k\}_k$ is the dual basis in
$V^*$;
\item[(c)] the tensors $T_i$, viewed as morphisms $\C\to (V^*)^{\otimes m_i}\otimes V^{\otimes
n_i}$, and their adjoints.
\end{itemize}
In other words, any irreducible representation of $K$ appears as a
subrepresentation of a tensor product of the representations on $V$
and $V^*$, and any $K$-intertwiner between such tensor products can
be written as a linear combination of compositions of morphisms of
the form $\iota\otimes\dots\otimes\iota\otimes
f\otimes\iota\otimes\dots\otimes\iota$, where $f$ is a morphism from
the above list.
\end{theorem}

We stress that we consider Hermitian adjoints here. Recall also that the Hermitian scalar product on the dual space $V^*$ is defined so that the dual basis for an orthonormal basis in $V$ is orthonormal.

\bp Since any finite dimensional representation of $K$ is
unitarizable, we can equally well prove a similar statement in the
setting of C$^*$-tensor categories (see,
e.g.,~\cite[Chapter~2]{NT}).

Consider the monoidal subcategory of the category $\Hilb_f$ of
finite dimensional Hilbert spaces with objects the tensor products
of the spaces $V$ and $V^*$ (in all possible orders) and morphism
spaces generated by the morphisms in the formulation of the theorem.
Completing this category with respect to subobjects and finite
direct sums we get a C$^*$-tensor category $\CC$. Since $r$ and
$\bar r$ are morphisms in $\CC$, $V^*$ is dual to $V$ in $\CC$.
Hence~$\CC$ is rigid. Next, observe that the flip maps $V^*\otimes
V\to V\otimes V^*$, $V\otimes V^*\to V^*\otimes V$ and $V^*\otimes
V^*\to V^*\otimes V^*$ are morphisms in $\CC$, since they can be
expressed in terms of $\sigma$, $r$, $\bar r$ and their adjoints. It
follows that the flip maps on tensor products of Hilbert spaces
define a unitary symmetry on $\CC$. We have an obvious symmetric
unitary fiber functor $\CC\to\Hilb_f$. Hence, by the Tannaka--Krein
duality, $\CC$ is the category of finite dimensional unitary
representations of a compact group~$H$.

Since $\CC$ is generated as a C$^*$-tensor category by the object
$V$, the representation of $H$ on $V$ is faithful, so~$H$ can be
considered as a subgroup of the unitary group $U(V)$. Since the
morphisms $T_i$ are $H$-intertwiners, we then get $H\subset K$. This
implies that
$$
\Hom_K((V^*)^{\otimes m}\otimes V^{\otimes n},
(V^*)^{\otimes k}\otimes V^{\otimes l})\subset \Hom_H((V^*)^{\otimes
m}\otimes V^{\otimes n},(V^*)^{\otimes k}\otimes V^{\otimes l}).
$$
But by the definition of $\CC=\Rep H$ we also have the opposite
inclusion, so
$$
\Hom_K((V^*)^{\otimes m}\otimes V^{\otimes n},(V^*)^{\otimes
k}\otimes V^{\otimes l})=\Hom_H((V^*)^{\otimes m}\otimes V^{\otimes
n},(V^*)^{\otimes k}\otimes V^{\otimes l}).
$$
Finally, since the representation of $K$ on $V\oplus V^*$ is
faithful and self-dual, any irreducible representation of~$K$
appears as a subrepresentation of $(V\oplus V^*)^{\otimes n}$. \ep

\begin{remark}
1. It is not really necessary to work with C$^*$-categories. The main
point is that any $*$-algebra of operators on a finite dimensional
Hilbert space is semisimple. This implies that the Karoubi envelope
of the category with objects the tensor products of the spaces $V$
and $V^*$ and morphisms as in the proof of the theorem is a rigid
semisimple monoidal category.

2. The tensors $T_i$ can be viewed as morphisms $V^{\otimes m_i}\to V^{\otimes n_i}$, and we could equally well use this interpretation in part (c) of the theorem, since the two ways of considering $T_i$ as morphisms are related by the duality morphisms $r$ and $\bar r$ and their adjoints.

3. While this paper was being written, N. Snyder informed us that a similar result to Theorem \ref{thm:TKu} appears independently in the work in progress by S. Morrison, N. Snyder and D. Thurston.
\end{remark}

A similar result holds for subgroups of compact orthogonal and symplectic groups, or, in other words, for representation categories generated by real and quaternionic representations. Namely, we have the following.

\begin{theorem}\label{thm:TK}
Assume $V$ is a finite dimensional complex Hilbert space with a real or quaternionic structure, so we are given an antilinear isometry $J\colon V\to V$ such that $J^2=1$ or $J^2=-1$.
Assume also that we are given a collection of tensors $T_i\in
(V^*)^{\otimes m_i}\otimes V^{\otimes n_i}$, $i\in I$. Let
$K$ be the compact group of unitary transformations of $V$ stabilizing all these
tensors and commuting with $J$. Then the monoidal category $\Rep K$ of finite dimensional
complex representations of $K$ is generated by $V$ and the following
morphisms:
\begin{itemize}
\item[(a)] the flip $\sigma\colon V\otimes V\to V\otimes V$;
\item[(b)] the morphism $s\colon \C\to V\otimes V$, $s(1)=\sum_k e_k\otimes J
e_k$, and its adjoint, where $\{e_k\}_k$ is an orthonormal basis in $V$;
\item[(c)] the tensors $T_i$, viewed as morphisms $V^{\otimes m_i}\to V^{\otimes
n_i}$, and their adjoints.
\end{itemize}
In other words, any irreducible representation of $K$ appears as a
subrepresentation of $V^{\otimes n}$, and any $K$-intertwiner
$V^{\otimes m}\to V^{\otimes n}$ can be written as a linear
combination of compositions of morphisms of the form
$\iota\otimes\dots\otimes\iota\otimes
f\otimes\iota\otimes\dots\otimes\iota$, where $f$ is a morphism from
the above list.
\end{theorem}

\bp The proof is similar to that in the unitary case. Briefly, the
Karoubi envelope of the category with objects $V^{\otimes n}$ and
morphism spaces generated by the morphisms in the formulation of the
theorem is the representation category of a closed subgroup
$H\subset U(V)$. Since $s$ is an $H$-intertwiner, the elements
of~$H$ commute with $J$. As in the unitary case, we then conclude
that the spaces of $K$- and $H$-intertwiners coincide. \ep

Theorems~\ref{thm:TKu} and~\ref{thm:TK} allow us, in principle, to
find different generators of representation categories of compact
Lie groups, or, equivalently, of reductive complex Lie groups. In fact, it follows from~\cite[Theorem~15.1]{GG} that if
we start with a compact connected simple Lie group $K$ and a
faithful irreducible unitary representation of $K$ on $V$, then,
apart from a few cases, we can find a tensor $T\in S^kV^*$ such that
$K$ coincides with the connected component of the identity of the
stabilizer of $T$ in $U(V)$. For some exceptional Lie groups, their
realizations as stabilizers of homogeneous polynomials on
representations of small dimensions can be found in~\cite{SV}. For
example, the compact simply connected group of type $E_6$ has
dimension $78$, while its smallest nontrivial irreducible
representation $V$ has dimension~$27$. It is known that there is a
unique up to a scalar factor homogeneous cubic  polynomial on $V$
invariant under $E_6$, and then $E_6$ is exactly the group of
unitary transformations stabilizing this polynomial.

\smallskip

Similar descriptions of representation categories can then be obtained for the $q$-deformations of reductive complex Lie groups $G$, at least for generic values of the deformation parameter or when $q$ is an indeterminate. For this we have to deform morphisms of $G$-modules. Let us explain, or rather remind, in detail how to do this.

In order to simplify the discussion, let us consider only the simply connected semisimple case. So, let~$G$ be a simply connected semisimple complex Lie group with Lie algebra $\g$. Fix a maximal torus and a system of simple roots. Consider the quantized universal enveloping algebra $\Uq$ over~$\Q(q)$, where $q$ is an indeterminate. Denote by $\Uqres\subset\Uq$
Lusztig's restricted integral form, but with the scalars extended
from $\Z[q,q^{-1}]$ to $\Qq$, see~\cite{L}
or~\cite[Section~9.3]{CP}.

Recall that a representation of $\Uq$ over $\Q(q)$ is called
admissible, or of type $1$, if it decomposes into a direct sum of
weight spaces.
Every finite dimensional admissible $\Uq$-module admits a
$\Uqres$-stable $\Qq$-lattice~$M$. Then on the specialization
$M_1=M/(q-1)M$ of $M$ to $q=1$ we get a representation of the
rational form $\g_\Q$ of $\g$ generated over~$\Q$ by the Chevalley generators of $\g$. Up to isomorphism this module does not depend on the
choice of~$M$. If $V$ is simple, with a highest weight vector $\xi$
of weight $\lambda$, then taking $M=(\Uqres)\xi$ we get that $M_1$
is also irreducible, with highest weight $\lambda$. It follows that
for arbitrary~$V$ and~$M$, the $\Uq$-module~$V$ decomposes into
simple highest weight modules in the same way as the
$\g_\Q$-module~$M_1$ does.

\begin{lemma} \label{lem:lift}
Let $V'$ and $V''$ be finite dimensional admissible $\Uq$-modules,
$M'\subset V'$ and $M''\subset V''$ be $\Uqres$-stable
$\Qq$-lattices. Then any morphism $M'_1\to M_1''$ of $\g_\Q$-modules
lifts to a morphism $M'\to M''$ of $\Uqres$-modules.
\end{lemma}

\bp The $\Qq$-module $\Hom_{\Uqres}(M',M'')$, being a submodule of a
free finite rank module, is itself free and of finite rank. Since
any morphism $V'\to V''$ of $\Uq$-modules can be multiplied by a
nonzero element of $\Qq$ to define a morphism $M'\to M''$ of
$\Uqres$-modules, the rank is equal to $
\dim_{\Q(q)}\Hom_{\Uq}(V',V'')=\dim_{\Q}\Hom_{\g_\Q}(M_1',M_1'').$
Therefore it suffices to show that if morphisms $T_i$, $1\le i\le
n$, form a $\Qq$-basis in $\Hom_{\Uqres}(M',M'')$, then their
specializations $T_{i1}\colon M_1'\to M_1''$ are linearly
independent over $\Q$.

Assume $\sum_i a_iT_{i1}=0$ for some numbers $a_i\in\Q$ not all of
which are zero. Then the specialization of $T=\sum_i a_iT_i$ to
$q=1$ is zero, so that $TM'\subset (q-1)M''$. But this implies that
$(q-1)^{-1}T$ defines a morphism $M'\to M''$, which contradicts the
assumption that the morphisms $T_i$ form a $\Qq$-basis in
$\Hom_{\Uqres}(M',M'')$. \ep

Assume now that $M_1$ is a $\g_\Q$-module which integrates to a
faithful representation of~$G$. Assume also that a compact form of $G$ is
the group of unitary transformations of $\C\otimes_\Q M_1$ (with
respect to a scalar product which is rational on $M_1$) commuting
with operators $1\otimes T_i\colon\C\otimes_\Q M_1^{\otimes m_i}\to
\C\otimes_\Q M_1^{\otimes n_i}$. Take any admissible
$\Uq$-module~$V$ with a $\Uqres$-stable $\Qq$-lattice $M$ such that its specialization to $q=1$ gives $M_1$. Then, using Lemma \ref{lem:lift}, we can
lift the morphisms $T_i$ to morphisms $T_{iq}\colon M^{\otimes
m_i}\to M^{\otimes n_i}$ of $\Uqres$-modules. Similarly for their
adjoints, as well as for the morphisms $r$, $\bar r$ and $\sigma$
(or more precisely, their analogues over $\Q$ for the space $M_1$)
and their adjoints.

Theorem~\ref{thm:TKu} implies that each space
$\Hom_{\g_\Q}(M_1^{\otimes m},M^{\otimes n}_1)$ has a basis over
$\Q$ consisting of compositions of morphisms
$\iota\otimes\dots\otimes\iota\otimes
f\otimes\iota\otimes\dots\otimes\iota$, where $f$ is one of the
morphisms $T_i$, $r$, $\bar r$, $\sigma$ or their adjoints. The same
compositions of morphisms with $f$ replaced by the corresponding
lifts $f_q$ to morphisms of $\Uqres$-modules are easily seen to be
linearly independent over $\Q(q)$. Since the admissible
$\Uq$-modules have the same fusion rules as in the classical case,
we conclude that these compositions span $\Hom_{\Uq}(V^{\otimes
m},V^{\otimes n})$ over $\Q(q)$. We thus see that the $\Q(q)$-linear
monoidal category of finite dimensional admissible $\Uq$-modules is
generated by $V$ and $V^*$ and the lifts of $T_i$, $r$, $\bar r$,
$\sigma$ and their adjoints.

We can also specialize the lifts of morphisms to $q\ne0$. At least for transcendental $q$ we then get generators of the category of complex finite dimensional admissible representations of quantized universal enveloping algebras.

If we know generators of the category of finite dimensional
admissible $\Uq$-modules, then, in turn, we can get generators and a
complete set of relations for the quantized algebra of regular
functions on $G$, since this algebra is the coend of the forgetful
fiber functor, and therefore it is generated by the matrix coefficients
of any generating set of modules and the relations are given by any
generating set of morphisms, see, e.g.,~\cite[Corollary~2.3.13]{Sc}.

\section{The main theorem}\label{s3}

In this section, following the strategy described above, we will find a particular presentation of quantized function algebras of automorphism groups of simple complex Lie algebras.

Let $G$ be a connected simple complex Lie group of adjoint type and
$\g$ be its Lie algebra. Fix a Cartan subalgebra $\h\subset\g$ and a system $(\alpha_k)_k$ of simple roots. Let $B\in\g^*\otimes\g^*$ be the
standard $\ad$-invariant symmetric form on $\g$, so that
$B(h_\alpha,h_\alpha)=2$ for every short root~$\alpha$, where
$h_\alpha\in \h$ is defined by $B(h_\alpha,h)=\alpha(h)$ for
$h\in\h$. Denote by $t\in\g\otimes\g$ the corresponding invariant
tensor.

As a compact form of $\g$ we take the real Lie algebra $\kk$ generated by the elements $ih_k$, $i(e_k+f_k)$, $e_k-f_k$, where $h_k,e_k,f_k$ are the Chevalley generators of $\g$, and denote by $K\subset G$ the corresponding compact form of $G$.
The form $B$ is negative definite on $\kk$, so $-B|_{\kk}$ extends to a positive definite Hermitian scalar product
on~$\g$. Then the adjoint of the morphism $B\colon\g\otimes\g\to\C$
is given by $B^*(1)=t$.

Next, consider the group $\Gamma$ of automorphisms of the Dynkin
diagram of $\g$. It acts by $B$-preserving automorphisms on~$\g$.
Together with the adjoint representation of $G$ this gives us a
faithful representation of the semidirect product $G\rtimes\Gamma$
on $\g$. Since $\Out(K)\cong\Gamma$, it follows that we can identify
the compact form $K\rtimes \Gamma$ of $G\rtimes\Gamma$ with the subgroup of the orthogonal group
$O(\kk,-B|_{\kk})$ stabilizing the Lie bracket
$L\in\g^*\otimes\g^*\otimes\g$. By Theorem~\ref{thm:TK} we therefore
get the following result.

\begin{proposition} \label{prop:TKad}
The monoidal category of finite dimensional complex representations
of $G\rtimes\Gamma$ is generated by $\g$ and the following
morphisms: the flip $\sigma\colon\g\otimes\g\to\g\otimes\g$,
$B\colon\g\otimes\g\to\C$, $t\colon\C\to\g\otimes\g$,
$L\colon\g\otimes\g\to\g$ and $L^*\colon\g\to\g\otimes\g$.
\end{proposition}

Let $V$ be a finite dimensional admissible $\Uq$-module which is a
deformation of the adjoint representation of $\g_\Q$, that is, it is
a simple module with highest weight equal to the maximal root
$\alpha_{\mathrm{max}}$. Fix a highest weight vector $\xi$ and put
$M=(\Uqres)\xi$. We identify $M_1$ with $\g_\Q$.

The group~$\Gamma$ acts on $\Uq$ by automorphisms $\alpha_\gamma$,
$\gamma\in\Gamma$, permuting the generators. We have a unique
representation of $\Gamma$ on $V$ which respects this action and such that its
specialization to $q=1$ gives the action of $\Gamma$ on $\g_\Q$. Namely,
this representation is given by $\gamma
(X\xi)=\chi(\gamma)\alpha_\gamma(X)\xi$ for $X\in\Uq$, where
$\chi\colon\Gamma\to\{-1,1\}$ is the same character as in the
classical case.

Now we can use Lemma~\ref{lem:lift} to lift the Lie bracket
$L|_{\g_\Q}\colon\g_\Q\otimes\g_\Q\to\g_\Q$ to a morphism
$$
L_q\colon M\otimes_{\Qq} M\to M
$$
of $\Uqres$-modules and then, by averaging over $\Gamma$, assume that $L_q$ is $\Gamma$-equivariant. The following shows that this determines $L_q$ in an essentially unique way.

\begin{lemma}
The Lie bracket is the unique up to a scalar factor $\Gamma$-equivariant morphism $\g\otimes\g\to\g$ of $\g$-modules.
\end{lemma}

\bp Recall that we assume that $\g$ is simple. It is known then,
see, e.g.,~\cite{KW} for a more general statement, that if $\g\not\cong\sln(\C)$ for $n\ge3$, then
$\g\otimes\g$ contains a single copy of $\g$, so in this case we do
not even need the $\Gamma$-equivariance to get uniqueness. On the other hand, if
$\g=\sln(\C)$ for some $n\ge3$, then the multiplicity of $\g$ in
$\g\otimes\g$ equals $2$: there is an extra copy of $\g$ in $S^2\g$.
Namely, a morphism $S^2\g\to\g$, or equivalently, a pairing between
$S^2\g$ and $\g$, comes from the nonzero invariant cubic form
$\Tr(X^3)$ on $\g=\sln(\C)$. Up to an inner automorphism, the only
nontrivial outer automorphism of $\sln(\C)$ is given by $X\mapsto
-X^t$. Since the cubic form is anti-invariant with respect to this
automorphism, we see that the corresponding morphism $S^2\g\to\g$ is
not $\Gamma$-equivariant. \ep

It follows that the space of $\Gamma$-equivariant morphisms
$V\otimes_{\Q(q)}V\to V$ of $\Uq$-modules is one-dimensional.  Hence
to get~$L_q$ we just have to take any such nonzero morphism and
multiply it by an element of $\Q(q)$ so that it maps $M\otimes_\Qq
M$ into $M$ and specializes to a nonzero morphism at $q=1$, and then
rescale it by a rational number to get the right specialization at
$q=1$. Furthermore, we may require $L_q$ to be indivisible in the
free rank one $\Qq$-module
$\Hom_{(\Uqres)\rtimes\Gamma}(M\otimes_{\Qq}M,M)$. This determines
$L_q$ uniquely up to a factor~$q^n$, $n\in\Z$.

\begin{remark}
Instead of the $\Gamma$-equivariance we could equivalently require $L_q$ to be zero on the quantum symmetric tensors. Namely, assume $\g=\sln(\C)$ for some $n\ge3$, which is the only case we have to take care of. Then on the isotypic component of $V\otimes_{\Q(q)}V$ corresponding to $V$ the braiding defined by the $R$-matrix has eigenvalues $\pm q^{-n}$ or $\pm q^n$, depending on which of the two standard $R$-matrices we take, see~~\cite[Corollary~23 in Section~8.4.3]{KS}. Then up to a scalar factor the morphism $L_q$ of $\Uq$-modules is specified by requiring it to kill the eigenvectors of the braiding with eigenvalue $q^{-n}$ or $q^{n}$.
Note also that explicit formulas for $L_q$ are known, see, e.g.,~\cite{DHGZ}.
\end{remark}

Similarly, we lift $B|_{\g_\Q}\colon \g_\Q\otimes\g_\Q\to\Q$ to a morphism
$$
B_q\colon M\otimes_{\Qq}M\to\Qq
$$
of $\Uqres$-modules. Since $V$ is simple and self-dual, the space of morphisms
$V\otimes_{\Q(q)}V\to\Q(q)$ is one-dimensional. Hence, again,~$B_q$ is unique up to a scalar factor, and if we  require $B_q$ to be indivisible in the free rank
one $\Qq$-module $\Hom_{\Uqres}(M\otimes_{\Qq}M,\Qq)$, then $B_q$ is unique up to a factor~$q^n$, $n\in\Z$. Clearly, $B_q$ is also $\Gamma$-equivariant.

In a similar way we can define $\Gamma$-equivariant lifts
$$
B'_q\colon \Qq\to M\otimes_\Qq M\ \ \text{and}\ \ L'_q\colon M\to M\otimes_\Qq M
$$
of $B^*|_\Q$ and $L^*|_{\g_\Q}$. Note that up to a factor, the morphism $B'_q$ is given by $1\mapsto\sum_k x^k\otimes
x_k$, where $(x_k)_k$ is a basis in $V$ and $(x^k)_k$ is the dual
basis: $B_q(x_k,x^l)=\delta_{kl}$.

\medskip

Next, consider the quantized algebra $\Pol_q(G)$ of regular
functions on $G$ over $\Q(q)$. By this we mean the algebra of matrix
coefficients of finite dimensional admissible $\Uq$-modules with
weights in the root lattice.~Put
$$
\Pol_q(G\rtimes\Gamma)=\Pol_q(G)\otimes_\Q\Pol(\Gamma;\Q),
$$
where $\Pol(\Gamma;\Q)$ is the algebra of $\Q$-valued functions on
$\Gamma$. Equivalently, $\Pol_q(G\rtimes\Gamma)$ is the algebra of
matrix coefficients of finite dimensional representations of
$\Uq\rtimes\Gamma$ over $\Q(q)$ such that their restrictions to
$\Uq$ have only weights in the root lattice. Denote by $\CC_q$ the
$\Q(q)$-linear monoidal category of such representations.

\begin{proposition}\label{prop:TKad2}
The monoidal category $\CC_q$ is semisimple and braided, with
braiding defined by the universal $R$-matrix ${\mathcal R}_q$ of
$\Uq$. Denote by $R_q$ the image of ${\mathcal R}_q$ under the
representation on $V\otimes_{\Q(q)}V$. Then $\CC_q$ is generated by
$V$ and the following morphisms: $\sigma R_q\colon V\otimes_{\Q(q)}
V\to V\otimes_{\Q(q)} V$, $B_q\colon V\otimes_{\Q(q)} V\to\Q(q)$,
$B'_q\colon \Q(q)\to V\otimes_{\Q(q)} V$, $L_q\colon
V\otimes_{\Q(q)} V\to V$, and $L'_q\colon V\to V\otimes_{\Q(q)}V$.
\end{proposition}

\bp The first statement follows from semisimplicity of the category
of finite dimensional admissible $\Uq$-modules using the standard
Clifford--Mackey type analysis of representations of crossed
products. Furthermore, the irreducible finite dimensional admissible
representations of $\Uq\rtimes\Gamma$ are classified by the
$\Gamma$-orbits of pairs $(\lambda,\pi)$, where $\lambda$ is a
dominant integral weight and $\pi\colon \Gamma_\lambda\to GL(H_\pi)$
is an irreducible representation (over $\Q(q)$) of the stabilizer
$\Gamma_\lambda$ of $\lambda$ in~$\Gamma$. The representation
corresponding to $(\lambda,\pi)$ is defined as follows. Consider the
irreducible $\Uq$-module $V_\lambda$ with a highest weight vector
$\xi_\lambda$ of weight $\lambda$. We have a representation
$\rho_\lambda$ of~$\Gamma_\lambda$ on $V_\lambda$ defined by
$\rho_\lambda(\gamma)X\xi_\lambda=\alpha_\gamma(X)\xi_\lambda$ for
$X\in\Uq$. The obvious representation of $\Uq$ on
$V_\lambda\otimes_{\Q(q)}H_\pi$ together with the representation
$\rho_\lambda\otimes\pi$ of $\Gamma_\lambda$ define an irreducible
representation of $\Uq\rtimes\Gamma_\lambda$. Finally, we induce
this representation to a representation of $\Uq\rtimes\Gamma$.

The statement about the braiding follows from the fact that the
universal $R$-matrix ${\mathcal R}_q$ is $\Gamma$-invariant. This,
in turn, can be checked either by looking at the explicit formula
for ${\mathcal R}_q$ or by using the characterization of ${\mathcal
R}_q$ by its action on the tensor product of lowest and highest
weight vectors.

Observe next that any irreducible representation of $\Gamma_\lambda$
over $\Q(q)$ has a unique up to isomorphism rational form. This can be seen either by noticing that the only nontrivial stabilizer groups we can get are~$S_2$ and~$S_3$, or by using the general fact that any
finite dimensional division algebra over $\Q$ remains a division
algebra after extending the scalars to $\Q(q)$. Therefore the simple
objects of $\CC_q$ are parametrized by the $\Gamma$-orbits of pairs
$(\lambda,\pi)$, where~$\lambda$ is a dominant weight in the root
lattice and $\pi$ is an irreducible representation of
$\Gamma_\lambda$ over $\Q$. The same set parametrizes the
irreducible finite dimensional representations of
$U(\g_\Q)\rtimes\Gamma$ over $\Q$ with weights in the root lattice.
It is clear that if we take the simple $\Uq\rtimes\Gamma$-module
$V'\in\CC_q$ corresponding to $(\lambda,\pi)$ and a
$\Uqres\rtimes\Gamma$-stable $\Qq$-lattice $M'\subset V'$, then the
specialization of~$M'$ to $q=1$ gives us a simple
$U(\g_\Q)\rtimes\Gamma$-module corresponding to $(\lambda,\pi)$.
This implies that the fusion rules in $\CC_q$ are the same as in the
representation category of $U(\g_\Q)\rtimes \Gamma$. We can also
conclude that if $V',V''\in\CC_q$ and $M'\subset V'$, $M''\subset
V''$ are $\Uqres\rtimes\Gamma$-stable $\Qq$-lattices, then
\begin{equation*}\label{eq:dim}
\dim_{\Q(q)}\Hom_{\Uq\rtimes\Gamma}(V',V'')
=\dim_\Q\Hom_{U(\g_\Q)\rtimes\Gamma}(M'_1,M''_1).
\end{equation*}

Now, if $W$ is a simple finite dimensional $U(\g_\Q)\rtimes\Gamma$
module with weights in the root lattice, then by
Proposition~\ref{prop:TKad} we know that the
$U\g\rtimes\Gamma$-module $\C\otimes_\Q W$ has a nonzero morphism
into $\g^{\otimes n}=\C\otimes_\Q(M_1^{\otimes n})$ for some $n$.
Hence $W$ embeds into $M_1^{\otimes n}$. By the above discussion it
follows that any simple object in $\CC_q$ embeds into $V^{\otimes
n}$ for some $n$.

Repeating the discussion at the end of the previous section, Proposition~\ref{prop:TKad} implies that the space
$\Hom_{U(\g_\Q)\rtimes\Gamma}(M_1^{\otimes m},M^{\otimes n}_1)$ has
a basis over $\Q$ consisting of compositions of morphisms
$\iota\otimes\dots\otimes\iota\otimes
f\otimes\iota\otimes\dots\otimes\iota$, where $f$ is the restriction
to the appropriate rational forms of one of the morphisms $\sigma$,
$B$, $B^*$, $L$, $L^*$. The same compositions of morphisms with $f$
replaced, resp., by $\sigma R_q$, $B_q$, $B'_q$, $L_q$, $L'_q$, are
then linearly independent over $\Q(q)$. Hence they span
$\Hom_{\Uq\rtimes\Gamma}(V^{\otimes m},V^{\otimes n})$. \ep

Fix a basis $(x_k)_k$ of the $\Qq$-lattice $M\subset V$. Let
$t_{kl}\in\Pol_q(G\rtimes\Gamma)$ be the matrix coefficients of the
representation of $\Uq\rtimes\Gamma$ on $V$ in this basis. In the
next theorem, which is our main result, we view~$R_q$ and~$L_q$ as matrices over~$\Qq$ with
respect to the bases $(x_k\otimes x_l)_{k,l}$ of $M\otimes_{\Qq}M$
and $(x_k)_k$ of $M$.

\begin{theorem}\label{mainth}
The $\Q(q)$-algebra $\Pol_q(G\rtimes\Gamma)$ is generated by the
coefficients of the matrix $T=(t_{kl})_{k,l}$, and the following is
a complete set of relations:
\begin{gather}
R_qT_1T_2=T_2T_1R_q,\label{eq1}\tag{R1}\\
A_q^{-1}T^tA_qT=1=TA_q^{-1}T^tA_q,\label{eq2}\tag{R2}\\
L_qT_1T_2=TL_q,\label{eq3}\tag{R3}
\end{gather}
where $A_q$ is the matrix $(B_q(x_k\otimes x_l))_{k,l}$.
\end{theorem}

\bp As we already discussed in the previous section, since the category $\CC_q$ is semisimple, the algebra of matrix
coefficients of $\Uq\rtimes\Gamma$-modules in $\CC_q$ is nothing
but the coend of the forgetful fiber functor on~$\CC_q$. Hence, by Proposition~\ref{prop:TKad2}, we conclude that
$\Pol_q(G\rtimes\Gamma)$ is generated by the elements $t_{kl}$ and a
complete set of relations is given by \eqref{eq1}, \eqref{eq3} and
\begin{gather}
B_qT_1T_2=B_q,\ \ T_1T_2B_q'=B_q',\label{eq2'}\tag{R2$'$}\\
T_1T_2L_q'=L_q'T,\label{eq4'}\tag{R4}
\end{gather}
where we again view $B'_q$ and $L'_q$ as matrices.

In terms of the matrix $A_q$ the first relation in \eqref{eq2'}
reads as $T^tA_qT=A_q$. Since up to a factor from $\Q(q)$ the
morphism $B'_q$ is given by
$$
1\mapsto \sum_k x^k\otimes x_k= \sum_{k,l}(A^{-1}_q)_{kl}x_k\otimes
x_l,
$$
the second relation in \eqref{eq2'} reads as
$TA^{-1}_qT^t=A^{-1}_q$. Thus relations \eqref{eq2} and \eqref{eq2'}
are equivalent.

It remains to show that \eqref{eq4'} follows from relations
\eqref{eq1}-\eqref{eq3}. For this we use the following description of~$L'_q$.
Equip $V\otimes_{\Q(q)}V$  with the nondegenerate invariant form
$$
B_q^{(2)}(x\otimes y,z\otimes w)=B_q(x\otimes w)B_q(y\otimes z).
$$
Then consider the adjoint $L^\dagger_q$ of $L_q$ with respect to the
invariant forms $B^{(2)}_q$ and $B_q$ on $V\otimes_{\Q(q)} V$ and
$V$, so $B^{(2)}_q(L^\dagger_qx,y\otimes z)=B_q(x\otimes L_q(y\otimes
z))$. In the classical case this would give us exactly the Hermitian
adjoint~$L^*$ we considered before. Hence by multiplying $L^\dagger_q$ by
a polynomial $f\in\Qq$ with value $1$ at $q=1$, we can get a
morphism $L'_q\colon M\to M\otimes_\Qq M$ such that its specialization to $q=1$ equals $L^*|_{\g_\Q}$. Since $B_q$ is $\Gamma$-equivariant, $L'_q$ is $\Gamma$-equivariant as well, so this is the required morphism.

Now, by taking adjoints (with respect to the forms $B^{(2)}_q$ and
$B_q$) of both sides of \eqref{eq3} we get
$$
(A^{-1}T^tA)_2(A^{-1}T^tA)_1L'_q=L'_q A^{-1}T^tA,
$$
where we used that the adjoint of an operator $C$ on $V$ is
$A^{-1}C^tA$ in our basis, and then that the adjoints of $T_1$ and
$T_2$ are $(A^{-1}T^tA)_2$ and $(A^{-1}T^tA)_1$, resp. In view of
\eqref{eq2}, the above identity is equivalent to~\eqref{eq4'}. \ep

Next we want to relax relation \eqref{eq2} to invertibility of $T$. This is possible by the following general result.

\begin{proposition}\label{prop:antipode}
Assume $\U$ is a quasitriangular Hopf algebra with invertible antipode and $R$-matrix $\RR$, and~$V$ is a finite dimensional $\U$-module. Fix a basis $v_1,\dots,v_n$ in $V$ and consider a universal unital algebra~$\A$ with generators~$t_{ij}$ and~$\tilde t_{ij}$ and relations
$$
RT_1T_2=T_2T_1R,\ \ T\tilde T=1=\tilde T T,
$$
where $T=(t_{ij})_{i,j}$, $\tilde T=(\tilde t_{ij})_{i,j}$, and $R$ is the matrix of the operator $\RR$ on $V\otimes V$ in the basis $(v_i\otimes v_k)_{i,k}$. Then $\A$ is a Hopf algebra with coproduct $\Delta(t_{ij})=\sum_k t_{ik}\otimes t_{kj}$, $\Delta(\tilde t_{ij})=\sum_k \tilde t_{kj}\otimes \tilde t_{ik}$ and antipode
$$
S(T)=\tilde T,\ \ S(\tilde T)=uTu^{-1},
$$
where $u$ is the Drinfeld element $m(S\otimes\iota)(\RR_{21})$ acting on $V$.
\end{proposition}

\bp
It is easy to see that the coproduct is well-defined. The opposite algebra is defined by the relations $RT_2T_1=T_1T_2R$ and $\tilde T^tT^t=1=T^t\tilde T^t$. From this we see that in order to prove that $S$ is a well-defined anti-homomorphism, it suffices to check that $(uTu^{-1})^t$ and $\tilde T^t$ are inverse to each other in $\A$. This will also imply that $S$ is indeed an antipode, since both $T$ and $\tilde T^t$ are corepresentations of $(\A,\Delta)$ and their entries generate $\A$.

Viewing $T$ and $T^{-1}=\tilde T$ as elements of $\End(V)\otimes\A$ rather than as matrices, say, $T=\sum_k c_k\otimes t_k$, $T^{-1}=\sum_i c_i'\otimes t_i'$, we thus have to check that
\begin{equation}\label{eq5}
\sum_{i,k} c_i'uc_ku^{-1}\otimes t_kt_i'=1=\sum_{i,k} uc_ku^{-1}c_i'\otimes t_i't_k.
\end{equation}

For this purpose take the relation $T_{23}^{-1}R_{12}T_{13}=T_{13}R_{12}T_{23}^{-1}$. Thus, if $\RR=\sum_ja_j\otimes b_j$, we have
$$
\sum_{i,j,k} a_jc_k\otimes c_i'b_j\otimes t_i't_k=\sum_{i,j,k} c_ka_j\otimes b_jc_i'\otimes t_kt_i'\ \ \text{in}\ \ \End(V)\otimes \End(V)\otimes\A.
$$
Since $\RR^{-1}=\sum_j S(a_j)\otimes b_j$, we have $\sum_{l,j} S^{-1}(a_l)a_j\otimes b_jb_l=1$, hence we get
$$
\sum_{i,k} c_k\otimes c_i'\otimes t_i't_k=\sum_{i,j,k,l} S^{-1}(a_l)c_ka_j\otimes b_jc_i'b_l\otimes t_kt_i'.
$$
Let us apply the operation $m_{21}$ to both sides of this identity. We get
$$
1=\sum_{i,j,k,l}  b_jc_i'b_lS^{-1}(a_l)c_ka_j\otimes t_kt_i',
$$
i.e., since $u=m(\iota\otimes S^{-1})(\RR_{21})$,
$$
1=\sum_{i,j,k}  b_jc_i'uc_ka_j\otimes t_kt_i'.
$$
Therefore, using that $\sum_{l,j} a_jS^{-1}(a_l)\otimes b_lb_j=1$, we get
$$
\sum_{i,k}  c_i'uc_k\otimes t_kt_i'=\sum_l b_lS^{-1}(a_l)\otimes1=u\otimes1,
$$
which gives the first identity in \eqref{eq5}.  The second identity is proved similarly, by starting with the relation $T_{23}R_{12}^{-1}T_{13}^{-1}=T_{13}^{-1}R_{12}^{-1}T_{23}$ and using that $u^{-1}=m(S^{-1}\otimes\iota)(\RR_{21}^{-1})$.
\ep

\begin{remark} \label{rem:Hopf}
The proof shows that the only properties of the invertible operator $R$ that we
need are that both $R$ and $R^{-1}$ are invertible in
$\End(V)^{\mathrm{op}}\otimes\End(V)$, and if $\sum_i{p_i\otimes
q_i}$ and $\sum_j r_j\otimes s_j$ are these inverses, then
$u=\sum_iq_ip_i$ and $v=\sum_js_jr_j$ are invertible in $\End(V)$.
Then the antipode is given by $S(T)=T^{-1}$,
$S(T^{-1})=uTu^{-1}=v^{-1}Tv$.
\end{remark}

We are now ready to prove the second version of our main result.

\begin{theorem}\label{mainth1a}
In the notation of Theorem~\ref{mainth}, the $\Q(q)$-algebra $\Pol_q(G\rtimes\Gamma)$ is generated by the
coefficients of $T$ and $T^{-1}$, and a complete set of relations is given by
$$
R_qT_1T_2=T_2T_1R_q,\ \ L_qT_1T_2=TL_q.
$$
\end{theorem}

\bp
Consider a universal algebra $\OO$ with generators and relations as in the formulation. In other words, $\OO$~is the quotient of the Hopf algebra $\A$ from Proposition~\ref{prop:antipode} by the additional relation $L_qT_1T_2=TL_q$. It is immediate that both the comultiplication and the antipode pass to the quotient, so $\OO$ is a Hopf algebra with antipode $S(T)=T^{-1}$, $S(T^{-1})=uTu^{-1}$, where $u$ is the Drinfeld element of $U_q\g$ acting on $V$. (Note that $U_q\g$ is not, strictly speaking, quasitriangular, as its $R$-matrix lives only in a completion of $U_q\g\otimes_{\Q(q)}U_q\g$, but the proof of Proposition~\ref{prop:antipode} goes through without any change; see also Remark~\ref{rem:Hopf}.)

The matrix $T$ defines a right $\OO$-comodule structure on $V$. To prove the theorem, it suffices to show that $B_q\colon V\otimes_{\Q(q)}V\to\Q(q)$ is a morphism of $\OO$-comodules, since, as we already used in the proof of Theorem~\ref{mainth}, this is equivalent to the identity $T^tA_qT=A_q$.

Consider the dual $\OO$-comodule $V^*$. We have the following standard morphisms of comodules:
\begin{align*}
& e\colon V\otimes_{\Q(q)}V^*\to \Q(q),\ \ x\otimes f\mapsto f(ux),\\
& i\colon \Q(q)\to V\otimes_{\Q(q)}V^*, \ \ 1\mapsto\sum_i x_i\otimes x^i,
\end{align*}
where $(x_i)_i$ is a basis in $V$ and $(x^i)_i$ is the dual basis. We can then define a morphism of $\OO$-comodules
$$
\tilde B_q=e(L_q\otimes\iota)(\iota\otimes L_q\otimes\iota)(\iota\otimes\iota\otimes i)\colon V\otimes_{\Q(q)}V\to\Q(q).
$$
Since $\Pol_q(G\rtimes\Gamma)$ is a quotient of $\OO$, this is also a morphism of $(U_q\g\rtimes\Gamma)$-modules. Hence it coincides with~$B_q$ up to a scalar factor. Therefore we only need to check that $\tilde B_q\ne0$.

For this we consider the specialization of $\tilde B_q$ to $q=1$. More precisely, recall that we fixed a $\Qq$-lattice $M$ in $V$ and identified $M_1$ with $\g_\Q$. As a lattice in $V^*$ we take $M^*=\Hom_{\Qq}(M,\Qq)$. Its specialization to $q=1$ is $\g_\Q^*=\Hom_\Q(\g_\Q,\Q)$. The morphism $\tilde B_q$ restricts to a morphism $M\otimes_\Qq M\to\Qq$ and this restriction is defined in exactly the same way as $\tilde B_q$, but using $M$ and $M^*$ instead of~$V$ and~$V^*$. Specializing this restriction to $q=1$ and using that then $L_q$ becomes the Lie bracket, while $u$ becomes equal to $1$, we see that the restriction of $\tilde B_q$ becomes the Killing form $X\otimes Y\mapsto\Tr((\ad X)(\ad Y))$ on $\g_\Q$.
\ep

\section{Extension to all \texorpdfstring{$q$}{q} except finitely many points}\label{s4}

As in the previous section, assume that $G$ is a connected simple complex Lie group of adjoint type. Denote the algebra $\Qq$ by $A$ and consider the quantized algebra of functions $\Pol_q^A(G)\subset\Pol_q(G)$ over~$A$ defined by Lusztig~\cite{L3}. We can then consider the subalgebra
$$
\Pol_q^A(G\rtimes\Gamma)=\Pol_q^A(G)\otimes_\Q\Pol(\Gamma;\Q)\subset\Pol_q(G\rtimes\Gamma)
$$
and specialize it to any nonzero $q_0\in\C$. We consider these specialized algebras over $\C$. Thus, we take the homomorphism $A\to\C$, $q\mapsto q_0$, and put
$$
\Pol_{q_0}(G\rtimes\Gamma)=\C\otimes_A \Pol_q^A(G\rtimes\Gamma).
$$
Note that it follows from the definition of $\Pol_q^A(G)$ in~\cite{L3} that for $q_0$ not a nontrivial root of unity the algebra $\Pol_{q_0}(G\rtimes\Gamma)$ can be defined in the same way as in the previous section but working over $\C$ instead of $\Q(q)$.

\begin{theorem}\label{mainth2}
Theorem~\ref{mainth1a} holds over $\C$ for all numerical values of $q$ except finitely many algebraic numbers. In fact, the following stronger statement holds: there is a polynomial $f\in\Q[q]$ with $f(1)\ne0$ such that for the ring $B=\Q[q,q^{-1},1/f]$ the $B$-algebra $\Pol_q^B(G\rtimes\Gamma)\subset\Pol_q(G\rtimes\Gamma)$ is described by the generators and relations as in Theorem~\ref{mainth1a}.
\end{theorem}

Recall that when $q$ is an indeterminate, then, by the discussion before Theorem~\ref{mainth}, we view~$R_q$ and~$L_q$ as matrices over~$\Qq$ with
respect to bases $(x_k\otimes x_l)_{k,l}$ of $M\otimes_{\Qq}M$
and $(x_k)_k$ of $M$. These matrices can be specialized to any nonzero complex number, and it is these specializations that we use in Theorem~\ref{mainth2}. For finitely many roots of unity we might even get zero matrices, but this is allowed by the formulation of the theorem. When we take values different from nontrivial roots of unity, then the specializations are never zero, and we conjecture that Theorem~\ref{mainth2} is true for all such values.

\begin{proof}
It is clear that Theorem~\ref{mainth1a} holds for all transcendental values of $q$, since we then have an inclusion $\Q(q)\hookrightarrow\C$. To extend the result to all but finitely many algebraic numbers, we argue as follows.

The generators and relations of Theorem \ref{mainth1a} make sense over the ring $A$. Let $\OO^A$ be the algebra defined by these generators and 
relations. We have a natural algebra homomorphism $\psi\colon \OO^A\to \Pol_q^A(G\rtimes\Gamma)$. Since the algebra $\Pol_q^A(G\rtimes\Gamma)$ is finitely generated by \cite[Proposition~3.3]{L3}, the homomorphism $\psi$ becomes surjective after inverting a polynomial $f_1(q)$. Using the fact that the specialization of $\Pol_q^A(G\rtimes\Gamma)$ to $q=1$ maps an $A$-basis into a basis of $\Pol^\Q(G\rtimes\Gamma)$ it is easy to see that we may assume that $f_1(1)\ne 0$. (For similar reasons we can actually assume that the roots of $f_1$ are nontrivial roots of unity.)

Let $\mathcal{K}=\operatorname{Ker}\psi$. By the arguments in the proof of \cite[Proposition~I.8.17]{BG}, the algebra $\OO^A$ is $\mathbb N$-filtered, and $\operatorname{gr}(\OO^A)$ is locally finite and strongly left Noetherian (more precisely, that proposition is proved when~$q$ is specialized to a numerical value, but it generalizes verbatim to our setting). Hence $\mathcal{K}$ is a finitely generated $\OO^A$-module. Thus, it follows from \cite[Theorem 0.3]{ASZ} that $\mathcal{K}$ is generically free over $\operatorname{Spec}A$. In other words,~$\mathcal{K}$~becomes free after inverting some polynomial $f_2(q)$. But Theorem \ref{mainth1a} implies that $\mathcal{K}\otimes_A \Q(q)=0$. Hence $\mathcal{K}$ becomes zero after inverting $f_2(q)$.

This means that the map $\psi$ becomes an isomorphism after inverting $f=f_1f_2$. In particular, we get an isomorphism if we specialize $q$ to any nonzero value except the roots of~$f$.

It remains to show that $f$ can be chosen so that $f(1)\ne 0$, or equivalently, that the localization of $\mathcal{K}$ to some Zariski neighbourhood of $q=1$ is zero. Clearly, we can choose a polynomial $f$ such that $f(1)\ne 0$, $\psi$ is surjective after inverting~$f$, and $(q-1)^N\mathcal{K}[1/f]=0$ for some $N\ge0$. Pick the smallest such $N$.

Assume that $N>0$. Pick $v\in \mathcal{K}[1/f]$ such that $(q-1)^{N-1}v\ne 0$ (it exists since $N$ was chosen the smallest possible). Then $v\notin (q-1)\mathcal{K}[1/f]$, i.e., ${\mathcal K}[1/f]/(q-1)\mathcal{K}[1/f]={\mathcal K}/(q-1){\mathcal K}\ne 0$. On the other hand, let us reduce the short exact sequence
$$
0\to \mathcal{K}[1/f]\to \OO^A[1/f]\to \Pol_q^A(G\rtimes\Gamma)[1/f]\to 0
$$
modulo $q-1$. Since $\Pol_q^A(G\rtimes\Gamma)$ is a free $A$-module, we have ${\rm Tor}^1_A(A/(q-1)A,\Pol_q^A(G\rtimes\Gamma)[1/f])=0$. Hence we get a short exact sequence
$$
0\to \mathcal{K}/(q-1)\mathcal{K}\to \OO^A/(q-1)\to \Pol^\Q(G\rtimes\Gamma)\to 0.
$$
But $\Pol^\Q(G\rtimes\Gamma)$ is defined by the specialization of the relations of Theorem \ref{mainth1a} to $q=1$, hence ${\mathcal K}/(q-1){\mathcal K}=0$. This is a contradiction. Thus $N=0$ and $\mathcal{K}[1/f]=0$, as desired.
\end{proof}

\begin{corollary}\label{mainth3} Outside of finitely many values of $q$ (not including $1$), every braided tensor autoequivalence of the category $\mathcal C$ of comodules over $\Pol_q(G\rtimes \Gamma)$ is trivial.
\end{corollary}

\begin{proof} Denote as before by $V$ a quantum deformation of the $\g$-module $\g$. It is not difficult to check using the arguments of the proof \cite[Theoreom 4.4(ii)]{DEN} that for every autoequivalence $F\colon \mathcal C\to \mathcal C$ one has $F( V)\cong V$. Also, by \cite[Proposition 3.12]{DEN}, Theorem~\ref{mainth2} implies that outside of finitely many values of $q$, every braided tensor functor out of $\mathcal C$ is determined by the image of $V$ and of the quantum Lie bracket $L_q\colon V\otimes V\to V$. Thus, every braided tensor functor $F\colon\mathcal C\to\mathcal C$ is determined by $\lambda\in \C^*$ such that $F(L_q)=\lambda L_q$. But by rescaling the isomorphism $V\cong F(V)$, we can make $\lambda=1$. This implies the corollary.
\end{proof}

This provides a new proof of \cite[Theorem~2.1]{NT0} (for all but finitely many $q$ not a nontrivial root of unity) and of \cite[Theorem~4.4(ii)]{DEN} for the groups $G_2$, $F_4$, $E_7^{\mathrm{ad}}$, $E_8$, for which $\Gamma=1$.
For $E_6$ there is a similar proof based on an analogue of Theorem \ref{mainth1a} representing quantum $E_6$ as the
symmetries of the unique (up to scaling) cubic form on the $27$-dimensional representation. Note that for the classical groups $SL(n)$, $SO(n)$, $Sp(2n)$, there is a direct argument in \cite[Theorem~4.6]{DEN} giving an even stronger statement, with a description of the set of forbidden values of $q$.

\end{document}